\documentclass[leqno]{amsart}
\usepackage{amssymb}
\usepackage{amsfonts,amsthm}
\usepackage{hyperref}

\newtheorem{theorem}{Theorem}[section]

\newtheorem{lemma}[theorem]{Lemma}

\theoremstyle{definition}
\newtheorem{definition}[theorem]{Definition}
\newtheorem*{acknowledgments}{Acknowledgments}
\theoremstyle{remark}
\newtheorem{remark}[theorem]{Remark}

\newtheorem{notation}{Notation}
\begin{document}
\title[Classification of  homogeneous Einstein metrics on pseudo-hyperbolic spaces]
{Classification of homogeneous Einstein metrics on
pseudo-hyperbolic spaces}
\author{Gabriel B\u adi\c toiu}
\address{``Simion Stoilow'' Institute of Mathematics  of the Romanian Academy, Research unit 4,
P.O. Box 1-764, 014700 Bucharest, Romania}
 \email{Gabriel.Baditoiu@imar.ro}
\date{\today}
\begin{abstract}
We classify the effective and transitive actions of a Lie group $G$ on an n-dimensional non-degenerate hyperboloid (also called  real pseudo-hyperbolic space), under the assumption that $G$ is a closed, connected Lie subgroup  of $SO_0(n-r,r+1)$, the connected component of the indefinite special orthogonal group. Assuming additionally  that $G$ acts completely reducible on $\mathbb R^{n+1}$, we also obtain that any $G$-homogeneous Einstein pseudo-Riemannian metric on a real, complex or quaternionic pseudo-hyperbolic space, or on a para-complex or para-quaternionic projective space is homothetic to either the canonical metric or the Einstein metric  of the canonical variation of a Hopf pseudo-Riemannian submersion.
\end{abstract}
\maketitle
\section{Introduction and the main theorem}
\noindent
The homogeneous Einstein Riemannian metrics on spheres and projective spaces are, up to homothety,  the canonical metrics or the Einstein metrics of the canonical variations of the  Hopf fibrations  (see Ziller \cite{ziller}). Essentially, up to a scaling factor, $S^{15}$ has 3 homogeneous Einstein Riemannian metrics, $S^{4n+3}$ (with $n\not=3$) and
 $\mathbb CP^{2n+1}$
have 2 homogeneous Einstein Riemannian metrics,
and each of the remaining spaces
 $S^{2n}$, $S^{4n+1}$, $\mathbb CP^{2n}$, $\mathbb HP^n$
has only one homogeneous Einstein metric (see Besse \cite[Theorem 9.86]{bes} and Ziller \cite{ziller}).

Motivated by the recent classification of the pseudo-Riemannian submersions with totally geodesic fibres from pseudo-hyperbolic spaces (see B\u adi\c toiu \cite{baditoiu}), in this paper we obtain a pseudo-Riemannian generalization of Ziller's classification mentioned above (see Ziller \cite{ziller}), and we prove the following main result.

\begin{theorem}\label{t:main}
Let $G$ be a connected, closed Lie subgroup of $SO_0(n-r,r+1)$ and assume that $G$ acts completely reducible on $\mathbb R^{n+1}$ and $r<n$.
Any $G$-homogeneous Einstein pseudo-Riemannian metric on one of the following sets: $H^{n}_r$, $\mathbb CH^{n/2}_{r/2}$,
$\mathbb HH^{n/4}_{r/4}$, $\mathbb AP^{n/2}$ (with $r+1=(n+1)/2$),  $\mathbb BP^{n/4}$ (with $r+1=(n+1)/2$)
is homothetic to either the canonical
metric or  the Einstein metric
  $g_{t_0}$ ($t_0\not=1$)
of the canonical variation of a Hopf pseudo-Riemannian submersion. Therefore, under the same assumption on $G$,  the following hold:
\begin{itemize}
\item[(i)]
  $H^{2m}_s$, $H^{4m+1}_s$, $\mathbb CH^{2m}_s$,
  $\mathbb CH^{2m+1}_{2s}$ (with $m\not=2s$), $\mathbb AP^{2m}$, $\mathbb BP^{m}$
 have only one homogeneous Einstein  metric;
\item[(ii)]
  $H^{4m+3}_{4s+3}$ (with $m\not=3$ and $m\not=2s+1$), $\mathbb CH^{4s+1}_{2s}$, $\mathbb CH^{2m+1}_{2s+1}$ (with $m\not=2s+1$)
 and $\mathbb AP^{2m+1}$ have 2 homogeneous Einstein  metrics.
\item[(iii)]
  $H^{8s+7}_{4s+3}$ (with $s\not=1$) and
  $\mathbb CH^{4s+3}_{2s+1}$
 have 3 homogeneous Einstein  metrics.
\item[(iv)]
  $H^{15}_7$ has 5 homogeneous Einstein  metrics.
\end{itemize}
\end{theorem}

The key ingredient of the proof of Theorem \ref{t:main} is the classification of effective transitive actions of a Lie group $G$ on a real pseudo-hyperbolic space under the assumption of Theorem \ref{t:main}.
Now, we give a short review of well-known classification results of effective and transitive actions.

The pioneering work
is due to Montgomery and Samelson (see \cite{ms}) and Borel (see \cite{borel}), who classified the compact Lie groups acting effectively and transitively on spheres.
Using homotopic methods, Onishchik obtained the classification of the connected compact Lie groups $G$ acting transitively on simply connected manifolds of rank 1 (see Onishchik \cite{oniscik,onitransitive}).

The case of transitive actions on non-compact spaces is more challenging than the compact case and therefore, one has to impose additional assumptions on the Lie group $G$. The case $G$ reductive was investigated by Onishchik in \cite{oni1}, where he studied the equivalent problem of finding decompositions $G=G'G''$ into two proper Lie subgroups $G'$ and $G''$. Of special interest for us is his classification of  semisimple decompositions of $so(n-r,r+1)$, simply because it solves our  problem of finding all transitive and effective actions on $H^n_s$ in the case of a semisimple $G\subset SO_0(n-r,r+1)$.
Using the Borel-Montgomery-Samelson classification of effective transitive actions on spheres, Wolf obtained a classification of the connected, closed Lie subgroups of $SO_0(n-r,r+1)$ acting transitively both on (a)
a component of a non-empty quadric $\{x\in\mathbb{R}^{n+1}_{r+1}\, |\, || x ||^2=a\}$ $(a\not = 0)$
and (b) the light cone
 $\{x\in\mathbb{R}^{n+1}_{r+1}\, |\, || x ||^2=0,  x\not= 0\}$ (see Wolf \cite[Theorem 3.1]{wolf}).
In our Theorem \ref{t:3.1}, we drop (b) and the  assumption on the semisimplicity of $G$. The proof of Theorem \ref{t:3.1}  uses Kramer's classification of Lie subgroup $G$ of $GL(m+1,\mathbb R)$ acting transitively on $\mathbb R^{m+1}\setminus\{0\}$ (see Kramer \cite{kramer}).
The classification of connected closed complex Lie subgroups of $SO(n,\mathbb C)$ acting transitively on the complex quadric $(v,v)=1$ was obtained by Kac \cite{kac}.

\section{The Hopf pseudo-Riemannian submersions and their canonical variations}
\noindent First, we introduce some standard definitions and notation that shall be needed throughout the paper.
\begin{definition}
   Let $\langle\cdot,\cdot\rangle_{\mathbb R^{n+1}_{r+1}}$
   be the standard inner product of signature $(n-r,r+1)$ on $\mathbb R^{n+1}$
   given by
\begin{eqnarray}
 \langle x,y\rangle_{\mathbb R^{n+1}_{r+1}}
 =-\sum\limits_{i=0}^rx_iy_i+\sum\limits_{i=r+1}^nx_iy_i
\end{eqnarray}
  for $x=(x_0,\cdot\cdot\cdot,x_n),y=(y_0,\cdot\cdot\cdot,y_n)\in\mathbb R^{n+1}$.
  For any $c<0$ and any positive integer $r\leq n$, the set
  $H^n_r(c)=\{x\in\mathbb R^{n+1}\ |\ \langle x,x\rangle_{\mathbb R^{n+1}_{r+1}}=1/c\}$
   is called the
 {\it real pseudo-hyperbolic space} of index $r$ and dimension $n$.
The hyperbolic space is defined as
 $H^n_0(c)=\{x=(x_0,x_1,\cdots,x_n)\in\mathbb R^{n+1}\ |\ x_0>0,\ \langle x,x\rangle_{\mathbb R^{n+1}_{1}}=1/c\}$.
For convenience, we write
 $H^n_r =H^n_r(-1)$.
\end{definition}
\begin{notation}
We define
 $$SO(n-r,r+1)=\{ g\in SL(n+1,\mathbb R) \ | \ \langle gx,gy\rangle_{\mathbb R^{n+1}_{r+1}}
              =\langle x,y\rangle_{\mathbb R^{n+1}_{r+1}} \}.$$
When $K$ is a Lie group, we shall always denote by $K_0$ its connected component of the identity.

Let $\mathbb C$, $\mathbb H$, $\mathbb A$, $\mathbb B$ be the algebras of complex, quaternionic, para-complex and para-quaternionic numbers, respectively. For $F\in\{\mathbb A,\mathbb B\}$, we denote  by $\bar{z}$, as usual, the conjugate of $z\in F$.
For any $z=(z_1,\cdots,z_m),w=(w_1,\cdots,w_m)\in F^m$, we define the standard inner product $\langle z,w\rangle_{F^m}$ on  $F^{m}$ by
\begin{eqnarray}
 \langle z,w\rangle_{F^m}=\mathrm{Re}(\sum\limits_{i=1}^m\bar{z_i}w_i).
\end{eqnarray}
The group
 $U^\pi(m)=\{g\in \mathrm{GL}(m,\mathbb A) \,|\, \langle gz,gw\rangle_{\mathbb A^m}=\langle z,w\rangle_{\mathbb A^m} \}$
is called the  para-unitary group (see \cite[Prop.~4]{cortes} or \cite[p.~508]{krahe}).
Let
 $Sp^\pi(m)=\{g\in \mathrm{GL}(m,\mathbb B) \,|\, \langle gz,gw\rangle_{\mathbb B^m}=\langle z,w\rangle_{\mathbb B^m} \}$
be the  para-symplectic group (see \cite[p.~510]{krahe}). We have a  natural inclusion
 $Sp^\pi(m)\subset U^\pi(2m)$ and some identifications $U^\pi(m)=\mathrm{GL}(m,\mathbb R)$ (see \cite[p.~508]{krahe}), and $Sp^\pi(m)\cong Sp(m,\mathbb R)$ (see \cite[p.~510, Prop.~1.4.3]{krahe}).  Here, our notation is different from \cite{krahe,kramer}, namely $Sp(m,\mathbb R)$ denotes the group of $2m\times 2m$-symplectic matrices with entries in $\mathbb R$.
\end{notation}
\begin{definition}
 We define (see \cite{baditoiu}, \cite{erdem,gm} for $\mathbb AP^m$, \cite{blazic} for $\mathbb BP^m$):
 $$\mathbb CH^m_s(c)=H^{2m+1}_{2s+1}(c/4)/U(1), \ \mathbb CH^m_s=\mathbb CH^m_s(-4), \ \mathbb HH^m_s(c)=H^{4m+3}_{4s+3}(c/4)/Sp(1),
 \ \mathbb HH^m_s=\mathbb HH^m_t(-4);
 $$
 $$\mathbb AP^m =\{z\in \mathbb A^{m+1} \,|\, \langle z,z\rangle_{\mathbb A^{m+1}}=1 \}/\{t=x+ey\in \mathbb{A} \,|\, t\bar{t}=1, x>0 \}=H^{2m+1}_{m}/H^1, (\text{with }e^2=1);$$
 $$\mathbb BP^m =\{z\in \mathbb B^{m+1} \,|\, \langle z,z\rangle_{\mathbb B^{m+1}}=1 \}/\{t\in \mathbb{B} \,|\, t\bar{t}=1\}=H^{4m+3}_{2m+1}/H^3_1.
 $$
\end{definition}
Any Hopf pseudo-Riemannian submersions can be written as a homogeneous map
 $\pi:G/K\to G/H$ with $K\subset H$
closed Lie subgroups in $G$
(see \cite{baditoiu}, for (10) see also Krahe \cite[p.~ 518, Example 2.2.1]{krahe}, for (7) see also Harvey \cite[p.~312]{harvey}):
\begin{itemize}
\item[(1)] $\pi_{\mathbb C}:H^{2m+1}_{2s+1}=SU(m-s,s+1)/SU(m-s,s)\to
       \mathbb CH^m_s=SU(m-s,s+1)/S(U(1)U(m-s,s)),$
\item[(2)]
    $\pi_{\mathbb A}:H^{2m+1}_{m}=SU^\pi(m+1)/SU^\pi(m)
    \to\mathbb AP^m=SU^\pi(m+1)/S(U^\pi(1)U^\pi(m)),$
\item[(3)]
    $\pi_{\mathbb H}:H^{4m+3}_{4s+3}=Sp(m-s,s+1)/Sp(m-s,s)\to
     \mathbb HH^m_s=Sp(m-s,s+1)/Sp(1)Sp(m-s,s),$
\item[(4)]
    $\pi_{\mathbb B}:H^{4m+3}_{2m+1}=Sp^\pi(m+1)/Sp^\pi(m)
    \to\mathbb BP^m=Sp^\pi(m+1)/(Sp^\pi(1)Sp^\pi(m)),$
\item[(5)]
    $\pi^1_{\mathbb O}:H^{15}_{15}=Spin(9)/Spin(7)\to
     H^8_8(-4)=Spin(9)/Spin(8),$
\item[(6)]
     $\pi^2_{\mathbb O}:H^{15}_7=Spin_0(8,1)/Spin(7)\to
     H^8(-4)=Spin_0(8,1)/Spin(8),$
\item[(7)]
    $\pi_{\mathbb O'}:H^{15}_{7}=Spin_0(5,4)/Spin_0(4,3)\to
    H^8_{4}(-4)=Spin_0(5,4)/Spin_0(4,4),$
\item[(8)]
    $\pi_{\mathbb C,\mathbb H}:\mathbb
     CH^{2m+1}_{2s+1}=Sp(m-s,s+1)/Sp(m-s,s)U(1)\to
    \mathbb HH^m_s=Sp(m-s,s+1)/Sp(m-s,s)Sp(1),$
\item[(9)]
    $\pi_{\mathbb C,\mathbb B}:\mathbb CH^{2m+1}_{m}=Sp^\pi(m+1)/Sp^\pi(m)U(1)
     \to\mathbb BP^{m}=Sp^\pi(m+1)/Sp^\pi(m)Sp^\pi(1),$
\item[(10)]
    $\pi_{\mathbb A,\mathbb B}:\mathbb
    AP^{2m+1}=Sp^\pi(m+1)/Sp^\pi(m)U^\pi(1)
    \to\mathbb BP^m=Sp^\pi(m+1)/Sp^\pi(m)Sp^\pi(1),$
\end{itemize}
here the pseudo-Riemannian metrics  on $H^m_s$ and  $H^m_s(-4)$  are the ones with constant curvature $c$, with $c=-1$ for $H^m_s$, and  $c=-4$ for $H^m_s(-4)$;  the pseudo-Riemannian metrics on $\mathbb CH^m_s$, $\mathbb HH^m_s$, $\mathbb AP^m$, $\mathbb BP^m$ are the ones with constant holomorphic, quaternionic, para-holomorphic  or para-quaternionic curvature $-4$; and we call these metrics the canonical ones.

\subsection{The Einstein metrics of the canonical variation}

Let $\pi:(M,g)\to (B,g')$ be a pseudo-Riemannian submersion. We denote  by
$\hat g$ the metrics induced on fibres. The family of metrics $g_t$, with $t\in\mathbb R\setminus\{0\}$  and $g_t$ given by
$$g_t=\pi^*g'+t\hat{g},$$
is called the {\it canonical variation} of $\pi$.
  To find the values of $t$ for which $g_t$ is an Einstein metric, we use the following
pseudo-Riemannian version of a theorem obtained in the Riemannian case by Matsuzawa \cite{matsuzawa}, and independently by Berard-Bergery, see
Besse \cite[Lemma 9.74]{bes}. First, we  introduce the notation: $\lambda'=s'/n$ and $\hat\lambda=\hat s/p$, where
$s'$ and $\hat s$ are the scalar curvatures of $g'$ and $\hat g$, respectively,
and $n=\dim M$ and $p=\dim fibre$.

\begin{lemma}\label{l:bb} Let $\pi:(M,g)\to (B,g')$ be a pseudo-Riemannian submersion
with totally geodesic fibres. Assume that $g$, $g'$ and $\hat g$ are
Einstein and  the O'Neill integrability tensor $A\not\equiv 0$. Then the following two conditions are
equivalent:
\begin{itemize}
\item[(i)]
$t_0=\frac{\hat\lambda}{\lambda'-\hat\lambda}$ is the unique nonzero
value different from $1$ such that $g_t$ is also Einstein
\item[(ii)]
 $\hat\lambda\not=\frac{1}{2}\lambda'$ and $\hat\lambda\not=0$.
\end{itemize}
\end{lemma}

\begin{remark} Note that $\hat\lambda=0$ when the fibres are one-dimensional. Therefore, the canonical variations of $\pi_{\mathbb
C}$ and $\pi_{\mathbb A}$ do not provide any non-canonical Einstein metrics
on the real pseudo-hyperbolic space.
\end{remark}
\begin{remark}
For the Hopf pseudo-Riemannian submersions (3-10), the value  $t_0\not=1$ for which $g_{t_0}$ is an Einstein metric is the following:
\begin{itemize}
\item[(a)] For $\pi_{\mathbb H}$ and $\pi_{\mathbb B}$, we see that
  $\lambda'=-(4m+8)$, $\hat\lambda=-2$, and hence $t_0=\frac{1}{2m+3}$.
\item[(b)] For $\pi^1_{\mathbb O}$, $\pi^2_{\mathbb O}$ and $\pi_{\mathbb
O'}$, we have
$\lambda'
=-28$, $\hat\lambda=-6$ which gives
$t_0=\frac{3}{11}$.
\item[(c)] For $\pi_{\mathbb C,\mathbb H}$,
               $\pi_{\mathbb C,\mathbb B}$,
               $\pi_{\mathbb A,\mathbb B}$
 we have $\lambda'=-(4m+8)$,
 $\hat\lambda =-4$ and thus
$t_0=\frac{1}{m+1}$.
\end{itemize}
Clearly, the Einstein metrics $g_{t_0}$ (with $t_0\not=1$) of the canonical variations of the Hopf pseudo-Riemannian submersions (3-10) are neither isometric  to each  other, nor to the canonical metrics.
\end{remark}
\begin{definition}
    The pseudo-Riemannian manifold $(M,g)$ is called a $G$\emph{-homogeneous manifold} if $G$ is a closed Lie subgroup of the isometry group $I(M,g)$.
\end{definition}
Note that  $(M,g)$ is  a $G$-homogeneous manifold if and only if    $G$ acts effectively and transitively on $M$ and $g$ is a $G$-invariant metric on $M$. To show that  all G-homogeneous Einstein  metrics are the ones claimed in Theorem \ref{t:main}, we shall first  classify the closed, connected groups $G\subset SO_0(n-r,r+1)$ acting effectively and transitively  on those spaces.

\section{The classification of the groups acting transitively on  pseudo-hyperbolic spaces.}

Throughout this section, we shall denote by $H$ the isotropy group of an action of $G$ on  $M$.
Let $Skew_{N}$ denote the additive group of skew-symmetric $N\times N$ matrices with real entries. Let
$$U_N=\left\{
   \left(
    \begin{array}{cc}
       I_{N} & C  \\
       0 &   I_{N}  \\
    \end{array}
  \right) \ \Big|\ C\in Skew_{N}\right\}, \ \ \
  U'_{2N}= \left\{
   \left(
    \begin{array}{cc}
       I_{2N} & C  \\
       0 &   I_{2N}  \\
    \end{array}
  \right) \ \Big|\ C\in Skew_{2N}, \det(C)=1\right\},
$$
where $I_N$ denotes the identity matrix.

In the next theorem, we classify Lie subgroups of $SO_0(n-r,r+1)$ acting transitively on $H^n_r$. We shall assume $r<n$. When $n=r$, $H^n_n=S^n$ and
the classification of the Lie groups acting transitively on spheres is well known (see Borel-Montgomery-Samelson  \cite{borel,ms} for the classification of the compact groups or the survey due to Gorbatsevich and Onishchik \cite{oni-book} for the non-compact case).

\begin{theorem}\label{t:3.1}
Assume that $r<n$ and that $G$ is a closed, connected subgroup
of $SO_0(n-r,r+1)$. The group $G$ acts effectively and transitively on
$H^{n}_{r}$ if and only if $G$ is contained in Table 1.
\begin{center}
 Table 1
  \begin{tabular}{| c | l | l | c |}
  \hline
   No.  & G             & H           & $G/H$
   \\    \hline
    (1) & $SO_0(n-r,r+1)$ & $SO_0(n-r,r)$ & $H^n_{r}$ \\\hline
    (2) & $Spin_0(8,1)$  & $Spin(7)$  & $H^{15}_{7}$\\\hline
    (3) & $Spin_0(5,4)$  & $Spin_0(4,3)$  & $H^{15}_{7}$\\\hline
    (4) & $Spin_0(4,3)$  & $G_2^*$  & $H^7_3$\\\hline
    (5) & $G_2^*$  & $SU(2,1)$  & $H^6_2$\\\hline
    (6) & $G_2^*$  & $SL(3,\mathbb R)=SU^\pi(3)$  & $H^6_3$\\\hline
    (7) & $SU(m-s,s+1)$ & $SU(m-s,s)$ & $H^{2m+1}_{2s+1}$\\\hline
    (8) & $Sp(m-s,s+1)$ & $Sp(m-s,s)$ & $H^{4m+3}_{4s+3}$\\\hline
    (9) & $Sp(m-s,s+1) Sp(1)$ & $Sp(m-s,s)  Sp(1)$ & $H^{4m+3}_{4s+3}$\\\hline
    (10) & $SL(m+1,\mathbb R)=SU^\pi(m+1)$  & $SL(m,\mathbb R)=SU^\pi(m)$  & $H^{2m+1}_m$\\\hline
    (11) & $Sp(m+1,\mathbb R)=Sp^\pi(m+1)$  & $Sp(m,\mathbb R)=Sp^\pi(m)$  & $H^{4m+3}_{2m+1}$ \\\hline
    (12) & $Sp(m+1,\mathbb R)Sp(1,\mathbb R)=Sp^\pi(m+1)  Sp^\pi(1)$
                & $Sp(m,\mathbb R)Sp(1,\mathbb R)=Sp^\pi(m)  Sp^\pi(1)$  & $H^{4m+3}_{2m+1}$ \\ \hline\
    (13) & $SU(m-s,s+1)U(1)=U(m-s,s+1)$  & $SU(m-s,s)U(1)=U(m-s,s)$  & $H^{2m+1}_{2s+1}$\\\hline
    (14) & $Sp(m-s,s+1)  U(1)$  & $Sp(m-s,s)  U(1)$  & $H^{4m+3}_{4s+3}$\\\hline
    (15) & $\mathrm{GL}_+(m+1)=SU^\pi(m+1)U^\pi_0(1)=U^\pi_0(m+1)$  & $\mathrm{GL}_+(m)=SU^\pi(m)U^\pi_0(1)=U^\pi_0(m)$  & $H^{2m+1}_m$\\\hline
    (16) & $Sp(m+1,\mathbb R)\mathbb R^*_+=Sp^\pi(m+1) U^\pi_0(1)$  & $Sp(m,\mathbb R)\mathbb R^*_+=Sp^\pi(m)  U^\pi_0(1)$  & $H^{4m+3}_{2m+1}$ \\\hline
    (17)& $U_{m+1}\rtimes  SU^\pi(m+1)$   &    & $H^{2m+1}_{m}$ \\\hline
    (18)& $U_{m+1}\rtimes  SU^\pi(m+1)U_0^\pi(1)$   &    & $H^{2m+1}_{m}$ \\\hline
    (19)& $U_{2(m+1)}\rtimes   Sp^\pi(m+1)$   &    & $H^{4m+3}_{2m+1}$ \\\hline
    (20)& $U'_{2(m+1)}\rtimes   Sp^\pi(m+1)$   &    & $H^{4m+3}_{2m+1}$ \\\hline
    (21)& $U_{2(m+1)}\rtimes   Sp^\pi(m+1)U_0^\pi(1)$   &    & $H^{4m+3}_{2m+1}$ \\\hline
    (22)& $U'_{2(m+1)}\rtimes   Sp^\pi(m+1)U_0^\pi(1)$   &    & $H^{4m+3}_{2m+1}$ \\\hline
    (23)& $\mathbb R\rtimes   Sp^\pi(m+1) U_0^\pi(1)$  &    & $H^{4m+3}_{2m+1}$ \\\hline
    (24)& $U_{m+1}\rtimes   SO(m+1)\mathbb R_+^*$  &    & $H^{2m+1}_{m}$ \\\hline
    (25)& $U_{8}\rtimes   Spin(7)\mathbb R_+^*$  &    & $H^{15}_{7}$ \\\hline
    (26)& $U_{7}\rtimes   G_2\mathbb R_+^*$  &    & $H^{13}_{6}$ \\\hline
    (27)& $\mathfrak g_2\rtimes   G_2\mathbb R_+^*$  &    & $H^{13}_{6}$ \\\hline
    (28)& $ad_{Im(\mathbb O)}\rtimes   G_2\mathbb R_+^*$  &    & $H^{13}_{6}$ \\\hline
  \end{tabular}
\end{center}
where $ad_{Im(\mathbb O)}$ and $\mathfrak g_2$ are the $7$-dimensional and the $21$-dimensional irreducible $G_2$-modules, respectively, of the decomposition $so( Im(\mathbb O))=\mathfrak g_2\oplus ad_{Im(\mathbb O)}$
(see \cite[Theorem 4.1]{baez}).
\end{theorem}

\begin{proof}
The group $G$ is not compact. Indeed, if $G$ is compact, then so are $H$ and $G/H=H^n_r$, and hence $n=r$.
We split the proof into two cases: (a) $G$  semi-simple and (b) $G$ non-semisimple.
\subsection{G  semisimple} In the semisimple case, we shall obtain the cases (2--12) of Table 1 from a classification theorem due to Onishchik (see \cite[Theorem 4.1]{oni1}). We first recall some facts on transitive actions from Onishchik \cite{oni1}.
Let $K'$ and $K''$ be two closed Lie subgroups of $K$,
and let $\mathfrak k'$, $\mathfrak k''\subset\mathfrak k$ be their associated Lie algebras.
The subgroup $K'$ acts transitively on $K/K''$ if and only if $K$ can be written as a product $K=K'K''$. In the case of a semisimple triple $(\mathfrak k, \mathfrak k' , \mathfrak k'')$, that is also equivalent to $\mathfrak k=\mathfrak k'+\mathfrak k''$.  Additionally,
one has $K/K''=K'/(K'\cap K'')$.
Specializing to our case, a closed, connected subgroup $G$ of $SO_0(n-r,r+1)$ acts transitively on $H^n_{r}=SO_0(n-r,r+1)/SO_0(n-r,r)$ if and only if $so(n-r,r+1)=so(n-r,r)+\mathfrak g$. By   Onishchik's classification of the semisimple decompositions of $so(n-r,r+1)$ with $r<n$ (see \cite[Theorem 4.1 and Table 1]{oni1}), we get the cases (2--12) in our table.

\subsection{G non-semisimple} We proceed by splitting this case into two subcases: (b1) $G$ acts irreducibly on  $\mathbb R^{n+1}$ and (b2) $G$ does not act irreducibly on  $\mathbb R^{n+1}$.
\subsubsection{G acts irreducibly on $\mathbb R^{n+1}$}\label{s:irr}
In \cite[p.~321, Theorem 6]{berger1}, Berger obtained that the subgroups $G$ of $\mathrm{GL}(n+1,\mathbb R)$ acting effectively and transitively  on $H^n_r$, with $G$ acting irreducible on $\mathbb R^{n+1}$, are, except a finite number, in the cases (1), (7-9), (13-14) of Table 1. To see that all excepted Lie groups are semisimple, we shall now recall  from Wolf \cite[Proof of Theorem 3.1]{wolf} the construction of the compact form $G^*$ associated to $G$ and his proof of the fact that   $G^*$ acts transitively  on a sphere.

If $G$ acts irreducibly on $\mathbb R^{n+1}$, then so does its Lie algebra $\mathfrak g\subset \mathrm{gl}(n+1,\mathbb R)$. Hence, by
\cite[Proposition 19.1]{humphreys}, $\mathfrak g$ is reductive and $\dim Z(\mathfrak g)\leq 1$, which correlated to our working assumptions: (i)
 $G\subset SO_0(n-r,r+1)$, and (ii) $G$  non-semisimple, it gives $G=(G,G)U(1)$ (see \cite[Lemma 1.2.1]{wolf}).
There exists a Cartan involution $T$ of $so(n-r,r+1)$ such that $\mathfrak g$ is $T$-invariant (see \cite[Theorem 6]{mostow}).

Let $H_x=\{g\in  SO_0(n-r,r+1)\, |\, gx=x\}\equiv SO_0(n-r,r)$ be the isotropy group at $x\in H^n_s$ and let $\mathfrak h$ be its Lie algebra. Changing $x$, we may assume that $\mathfrak h$ is also  $T$-invariant (see \cite[Theorem 6]{mostow}, or \cite{wolf}).   The transitivity of $G$ on $H^n_r=SO_0(n-r,r+1)/H_x$ simply  implies that $so(n-r,r+1)=\mathfrak g+\mathfrak h$.
Let
$$\mathfrak s_-=\{X\in so(n-r,r+1)|T(X)=-X\}, \ \ \ \mathfrak s_+=\{X\in so(n-r,r+1)|T(X)=X\},$$
$$\mathfrak g_\pm=\mathfrak s_\pm\cap \mathfrak g, \  \ \mathfrak h_\pm=\mathfrak s_\pm\cap \mathfrak h.$$
The  associated compact forms of $so(n-r,r+1)$, $\mathfrak g$ and $\mathfrak h$, defined by
 $$so(n-r,r+1)^*=\mathfrak s_+ + i  \mathfrak s_-, \ \ \mathfrak g^*=\mathfrak g_+ + i \mathfrak g_-, \ \ \
   \mathfrak h^*=\mathfrak h_++i \mathfrak h_-,$$
  naturally satisfy the relation
  $$so(n+1)=so(n-r,r+1)^*=\mathfrak g^*+\mathfrak h^*=\mathfrak g^*+so(n).$$
 Hence, the connected compact Lie group $G^*$ (with $\mathrm{Lie}(G^*)=\mathfrak g^*$) acts transitively and effectively on the sphere $S^{n}=SO(n+1)/SO(n)$ (see Wolf \cite[Proof of Theorem 3.1]{wolf}), and thus, the non-semisimple Lie groups $G^*$ belong the infinite families $U(m)$ or $Sp(m)U(1)$. It follows that $G$ must be one of the groups in the cases (13-14).

\subsubsection{G does not act irreducibly on $\mathbb R^{n+1}_{r+1}$}
Let $V$ be a proper G-invariant subspace of  $\mathbb R^{n+1}_{r+1}$. By Wolf \cite[Lemma 8.2]{wolf2}, we have that $2(r+1)\leq n+1$ and $W_1=V\cap V^\perp$ is a $G$-invariant maximal totally isotropic subspace of dimension $r+1$.

Let $W_2$ be a totally isotropic space such that $W_1\oplus W_2\oplus U=\mathbb R^{n+1}_{r+1}$, $\dim W_2=\dim W_1=r+1$,   $U^\perp=(W_1\oplus W_2)^\perp$ and $U$ does not contain any isotropic vector. The decomposition $W_1\oplus W_2\oplus U=\mathbb R^{n+1}_{r+1}$  is called a Witt decomposition (see \cite[p.~160, Exercise 9]{onishchik-book}).
 Let $Q$ be the quadratic form on $\mathbb R^{n+1}_{r+1}$ given in the standard basis by
$Q(x,y)=\langle x,y\rangle_{\mathbb R^{n+1}_{r+1}}$.
Clearly, there exists an orthonormal basis $\{e_1,\cdots, e_{n+1}\}$ of  $\mathbb R^{n+1}_{r+1}$ with $Q(e_i,e_i)=-1$ for $i\in\{1,\cdots,r+1\}$, $Q(e_j,e_j)=1$ for $j\in\{r+2,\cdots,n+1\}$ and such that $ \{ w_1,\cdots,w_{r+1}\}$ , $ \{ w_{r+1},\cdots,w_{2r+2}\}$, $ \{ w_{2r+3},\cdots,w_{n+1}\}$ are bases of $W_1$, $W_2$ and $U$ respectively, with
\begin{eqnarray*}
    w_i= \frac{1}{2}(e_i-e_{i+r+1}), \ \ \
  w_{i+r+1} = \frac{1}{2} (e_i+e_{i+r+1}), \ \ \
  w_{k} = e_k,
\end{eqnarray*}
for any $i\in\{1,\cdots, r+1\}$ and $k\in\{2r+3,\cdots, n+1\}$.
 Let
 $$
   \eta=
   \left(
    \begin{array}{ccc}
      -I_{r+1} & 0 & 0 \\
      0 & I_{r+1} & 0 \\
      0 & 0 & I_{n-2r-1} \\
    \end{array}
  \right)  \text{\ \ and\ \ }
  W=
   \left(
    \begin{array}{ccc}
      \frac{1}{2}I_{r+1} & \frac{1}{2}I_{r+1} & 0 \\
    -\frac{1}{2}I_{r+1} & \frac{1}{2}I_{r+1} & 0 \\
      0 & 0 & I_{n-2r-1} \\
    \end{array}
  \right).
  $$
 With respect to the basis $E=\{e_1,\cdots, e_{n+1}\}$, any $g^E\in SO_0(n-r,r+1)$ satisfies
$
  (g^E)^t \eta  g^E = \eta,
$
which, with respect to the basis $\{w_1,\cdots, w_{n+1}\}$, becomes
\begin{eqnarray}\label{eq:W}
g^t [(W^{-1})^t\eta W^{-1}]g=(W^{-1})^t\eta W^{-1}.
\end{eqnarray}
Any $g\in G\subset SO_0(n-r,r+1)$ is a linear transformation on  $\mathbb R^{n+1}_{r+1}$, which can be written with respect to the basis $\{w_1,\cdots,w_{n+1}\}$ in the form
$$
g=\left(
    \begin{array}{ccc}
      A & B_0 & D_0 \\
      0 & B & D_1 \\
      0 & B_1 & D \\
    \end{array}
  \right),
$$
with $A\in \mathrm{GL}(r+1,\mathbb R)$, $B, B_0\in \mathrm{M}(r+1,\mathbb R)$,  $D_0,D_1\in M(r,n-2r-1,\mathbb R)$, $B_1\in M(n-2r-1,r,\mathbb R)$ and  $D\in \mathrm{M}(n-2r-1,\mathbb R)$.
  By \eqref{eq:W},
  we easily get that
\begin{eqnarray}\label{eq:o}
B=(A^{-1})^t, \ \ D_1=0,\ \   D\in O(n-2r-1).
\end{eqnarray}
\begin{eqnarray}\label{eq:B0skew}
2B^tB_0+2B_0^tB+B_1^tB_1=0
\end{eqnarray}
\begin{eqnarray}\label{eq:b1}
2B^tD_0+B_1^tD=0.
\end{eqnarray}
  It follows that $B_1=-2DD_0^tB$, and, hence, $B_1^tB_1=4 B^t D_0 D_0^t B$.  Let $B_0'=B_0B^{-1}$. By \eqref{eq:B0skew}, we get that
  \begin{eqnarray}\label{eq:B0skew2}
  B_0'+B_0'^t=-2D_0D_0^t
  \end{eqnarray}

\begin{lemma}
The transitivity of $G$ implies $n=2r+1$.
\end{lemma}
\begin{proof}
Suppose that $n>2r+1$.
For any $z \in\mathbb R^{r+1}\setminus\{0\}$ and any  matrices $A$,$B$,$B_1$,$B_0'$,$D$ as above, satisfying the conditions \eqref{eq:o},\eqref{eq:B0skew},\eqref{eq:b1},\eqref{eq:B0skew2}, set
$$
 \left(
    \begin{array}{c}
      y_1   \\
      y_2  \\
      y_3  \\
    \end{array}
  \right)
  =
  \left(
    \begin{array}{ccc}
      A & B_0'B & D_0 \\
      0 & B & D_1 \\
      0 & B_1 & D \\
    \end{array}
  \right)
  \left(
    \begin{array}{c}
     -\frac{z}{R}  \\
      z  \\
      0  \\
    \end{array}
  \right).
$$

With respect to the basis $\{w_i\}$,  the quadratic $Q$ writes as
\begin{eqnarray}
Q(x_1,\cdots,x_{2r+2},x_{2r+3},\cdots,x_{n+1} )
=x_1x_{r+2}+\cdots+x_{r+1}x_{2r+2}+x_{2r+3}^2\cdots+x_{n+1}^2.
\end{eqnarray}
Hence, for any $z=(z_1,\cdots,z_{r+1})\in\mathbb R^{r+1}\setminus\{0\}$, we obviously have
$Q(-\frac{z_1}{R},\cdots,-\frac{z_{r+1}}{R},z_1,\cdots,z_{r+1},0,\cdots,0)=-1$
with $R= {z_1^2+\cdots+z_{r+1}^2}$. Since $G$ acts on $H^n_r$, we have
\begin{eqnarray*}
-1&=&Q(y_1,y_2,y_3)=y_1^ty_2+y_3^ty_3
=(-\frac{1}{R}z^tA^t+z^tB^tB_0'^t)(Bz)+ z^tB_1^tB_1z\\
&=&-\frac{1}{R}z^tA^tBz+ z^t(B^tB_0'^tB+4B^tD_0D_0^tB)z
=-\frac{1}{R}z^tz+ z^tB^t(B_0'^t-2(B_0'+B_0'^t))Bz\\
&=&-1+(Bz)^t(-2 B_0'-B_0'^t)Bz .
\end{eqnarray*}
In particular, for  $z=B^{-1}e_i$ or $z=B^{-1}e_j$ or $z=B^{-1}(e_i+e_j)$, we get
$$0=e_i^t(2 B_0'+B_0'^t)e_i, \ \ 0=e_j^t(2 B_0'+B_0'^t)e_j \ \text{and}$$
$$0=(e_i+e_j)^t(-2 B_0'-B_0'^t)(e_i+e_j)
  = (-2 B_0'-B_0'^t)_{ij}+ (-2 B_0'-B_0'^t)_{ji}=-3(B_0'+B_0'^t)_{ij} ,$$
hence, by \eqref{eq:B0skew2}, $D_0D_0^t=0$, and thus $D_0=0$ and $B_1=0$. It follows that $y_3=B_1z=0$.

 Therefore $G.(-z/R,z,0)\subset H^{2r+1}_r\not=H^{n}_{r}$,
 which is a contradiction to the transitivity of $G$ on $H^{n}_{r}$.
\end{proof}
Any $g\in G\subset SO_0(r+1,r+1)$ can be written as
$$
g=
  \left(
    \begin{array}{cc}
      A & B_0'B  \\
      0 & B   \\
    \end{array}
  \right)
$$
for some skew-symmetric matrix $B_0'$  and some $A\in GL_+(r+1,\mathbb R)$ with $B=(A^{-1})^t$.
Let
$$G_0=\left\{
  \left(
    \begin{array}{cc}
      A & 0 \\
      0 & (A^{-1})^t \\
    \end{array}
  \right)
  \in SU^\pi(r+1) \Big|
  \text{there exists }  B_0'\in Skew_{r+1}\ \text{such that }
   \left(
    \begin{array}{cc}
      A & B_0'(A^{-1})^t \\
      0 & (A^{-1})^t \\
    \end{array}
  \right)\in G\
  \right\}.
  $$
 $$G_1=\left\{
  \left(
    \begin{array}{cc}
      I_{r+1} & B_0' \\
      0 & I_{r+1} \\
    \end{array}
  \right)
   \Big| B_0'\in Skew_{r+1} \text{\& there exists } A\in GL(r+1,\mathbb R) \ \text{such that }
   \left(
    \begin{array}{cc}
      A & B_0'(A^{-1})^t \\
      0 & (A^{-1})^t \\
    \end{array}
  \right)\in G
  \right\}.
  $$
We see that $G_0$, $G_1$ are subgroups in $G$; $G_1$ is normal in $G$ and $G=G_1G_0$.
For any $$h=\left(
    \begin{array}{cc}
      A & 0 \\
      0 & (A^{-1})^t \\
    \end{array}
  \right)\in G_0,$$
   let  $\theta_h:G_1\to G_1$ given by $\theta_h(B_0')=AB_0'A^t$.
It follows that $G=G_1\rtimes G_0$ is the semidirect product of $G_1$ and $G_0$.

Since
$Q(-\frac{y_1}{R},\cdots,-\frac{y_{r+1}}{R},y_1,\cdots,y_{r+1})=-1$
with $R= {y_1^2+\cdots+y_{r+1}^2}$, for any $(y_1,\cdots,y_{r+1})\in\mathbb R^{r+1}\setminus\{0\}$,
it follows that $G_0$ acts effectively and transitively on $\mathbb R^{r+1}\setminus\{0\}$.
Therefore, the Lie group $G_0\subset \mathrm{GL}(r+1,\mathbb R)$ acts irreducibly on $\mathbb R^{r+1}$ and, hence, by
\cite[Proposition 19.1]{humphreys} its Lie algebra $\mathfrak g_0$ is reductive and $\dim Z(\mathfrak g_0)\leq 1$.
Thus, $G_0$ is reductive and $\dim Z(G_0)^0\leq 1$. The action $\theta$ of the reductive Lie group $G_0$ on $G_1$ induces a completely reducible representation of $G_0\to End(\mathfrak g_1)$.
In \cite{kramer}, Kramer classified the closed connected Lie subgroups $G_0\subset GL(r+1,\mathbb R)$ acting transitively on $\mathbb R^{r+1}\setminus\{0\}$. By \cite[Theorem 6.17]{kramer},
$G_0\subset GL(r+1,\mathbb R)$ is, up to conjugation, one the following:
\begin{itemize}
\item[(a)]  $SL(m+1,\mathbb R)$, $SL(m+1,\mathbb R)\mathbb R^*_+$,
    $Sp(m+1,\mathbb R)$, $Sp(m+1,\mathbb R)\mathbb R^*_+$, or
\item[(b)]
$SL(m+1,\mathbb C)$, $SL(m+1,\mathbb C) U(1)$, $SL(m+1,\mathbb C) S_a$, $GL(m+1,\mathbb C)$,
\item[(c)]
$SL(m+1,\mathbb H)$, $SL(m+1,\mathbb H) U(1)$, $SL(m+1,\mathbb H) S_a$, $SL(m+1,\mathbb H) \mathbb C^*$,
\item[(d)]
$Sp(m+1,\mathbb C)$, $Sp(m+1,\mathbb C) U(1)$, $Sp(m+1,\mathbb C) S_a$,
$Sp(m+1,\mathbb C) \mathbb C^*$,
\item[(e)]
  $SU(m+1)S_a$, $SU(m+1)\mathbb C^*$ , $Sp(m+1)S_a$, $Sp(m+1)\mathbb C^*$,
\item[(f)]
$Spin(9,1), Spin(9,1) \mathbb R^*_+$,  $Spin(9)\mathbb R^*_+$,
\item[(g)]
$SL(m+1,\mathbb H) Sp(1)$, $SL(m+1,\mathbb H) \mathbb H^*$,   $Sp(m+1)\mathbb H^*$,
\item[(h)]
 $SO(m+1)\mathbb R^*_+$,
 $Spin(7)\mathbb R^*_+$, $G_2\mathbb R^*_+$,

\end{itemize}
when $r\geq 2$; or $G_0=\mathbb R^*$ when $r=0$; or $G_0\in\{ \mathbb C^*, SL(2,\mathbb R), SL(2,\mathbb R)\mathbb R^*_+\}$ when $r=1$. Here $S_a=\{e^{t(1+ia)|t\in\mathbb R}\}$ is a 1-dimensional subgroup of $\mathbb H^*$.  Note that in the case $r=1$,
we can write $\mathbb C^*= SO(2)\mathbb R^*_+$.

Since $SL(m+1,\mathbb R)$ acts irreducibly on $\Lambda^2\mathbb R^{m+1}=Lie(U_{m+1})$,
we get $G_1\in\{U_{m+1},\{e\}\}$ when $G_0\in\{SL(m+1,\mathbb R),SL(m+1,\mathbb R)\mathbb R^*_+\}$; and these  correspond  to (10),(15) and (17-18) in  Table 1.

Analogously to the complex quadric (see Kac \cite[\S 3]{kac}), the action of $Sp(m+1,\mathbb R)$ on $\Lambda^2\mathbb R^{2m+2}=Skew_{2m+2}$ gives a natural decomposition of $\Lambda^2\mathbb R^{2m+2}$ into two irreducible $Sp(m+1,\mathbb R)-$modules:
$Skew_{2m+2}=Lie(U'_{2m+2})\oplus\mathbb R$. Therefore
$G_1\in\{U_{2m+2},U'_{2m+2},\mathbb R,\{e\}\}$ for $G_0\in\{Sp(m+1,\mathbb R), Sp(m+1,\mathbb R)\mathbb R^*_+\}$. Identifying
$$\mathbb R=\left\{\left(
    \begin{array}{cc}
      1 & t \\
      0 & 1\\
    \end{array}
  \right)\Big|t\in\mathbb R\right\}\text{ and }
 \mathbb R^*_+=\left\{\left(
    \begin{array}{cc}
      \lambda & 0 \\
      0 & \lambda^{-1}\\
    \end{array}
  \right)\Big|\lambda\in\mathbb R^*_+\right\}=U^\pi_0(1),
$$
we see that $\mathbb R$ commutes with
$Sp^\pi(m+1) =Sp(m+1,\mathbb R)$ and $\mathbb R$ is normal in $G$.
When $G_0\in\{Sp(m+1,\mathbb R), Sp(m+1,\mathbb R)\mathbb R^*_+\}$, we get the entries (11), (16), (19-23) of Table 1.

\begin{lemma}\label{GL}
The semidirect product group $G_1\rtimes GL(m+1,\mathbb C)$ does not act transitively on $H^{4m+3}_{2m+1}$.
\end{lemma}
\begin{proof}
Let $C$ be the complex quadric $z_1^2+\cdots+z^2_{2m+2}=1$, $z=(z_1,\cdots,z_{2m+2})\in\mathbb C^{2m+2}$.
 We identify
 $\mathbb C^{2m+2}\ni z=(Re(z),Im(z))\in\mathbb R^{4m+4}$.
 We introduce the notation $z_k=x_k+iy_k$,
$t_k=y_k-x_k$ and $q_k =y_k+x_k$, for any $k\in\{1,\cdots,2m+2\}$.
Clearly, $z\in C$ if and only if
\begin{eqnarray}
 \sum_{k=1}^{2m+2} t_kq_k &=&-1 \ \ \text{and} \\
 \sum_{k=1}^{2m+2} t_k^2&=&\sum_{k=1}^{2m+2} q_k ^2 . \label{eq2}
\end{eqnarray}
  With respect to basis $\{w_k\}$,  our real quadric $H^{4m+3}_{2m+1}$ is given by
$\sum t_kq_k =-1$.

We embed $G_0=GL(m+1,\mathbb C)$ into $GL(2m+2,\mathbb R)$, and denote its image also by $GL(m+1,\mathbb C)$.
Let $(t',q')\in C\subset H^{4m+3}_{2m+1}$, $g\in G_1\rtimes GL(m+1,\mathbb C)$ and
$\left(
    \begin{array}{c} a\\b \\
    \end{array}
    \right)=g\left(
    \begin{array}{c}t' \\ q'\\
    \end{array}
    \right).
    $
Any $g\in G_1\rtimes GL(m+1,\mathbb C)$ can be written as
$$g
=
\left(
    \begin{array}{cc}
      I_{2m+2} & B_0' \\
      0 & I_{2m+2} \\
    \end{array}
  \right)
  \left(
    \begin{array}{cc}
      A & 0 \\
      0 & B\\
    \end{array}
  \right),
$$
with $A\in GL(m+1,\mathbb C)$, $B=(A^t)^{-1}$ and $B_0'+B_0'^t=0$.
Since $GL(m+1,\mathbb C)$ acts transitively on $C$, it follows that
$$
\left(
    \begin{array}{c} a\\b \\
    \end{array}
    \right)
=
\left(
    \begin{array}{cc}
      I_{2m+2} & B_0' \\
      0 & I_{2m+2} \\
    \end{array}
  \right)
    \left(
    \begin{array}{c}t \\ q\\
    \end{array}
    \right),
$$
for some $(t, q)\in C$. Then $q=b$ and $t=a-B_0'b$. By \eqref{eq2}, we get
 $$(a-B_0'b)^t(a-B_0'b)=b^tb.$$

To see that $G$ does not act transitively it is sufficient to show that there exists  $(a,b)\in\mathbb R^{4m+4}\setminus\{0\}$ with $b^ta=-1$ such that
 $$f(a,b,B_0')=(a-B_0'b)^t(a-B_0'b)-b^tb>0$$
 for any skew-symmetric matrix $B_0'$.
 Any skew-symmetric matrix $B'_0$ can be written as $B_0'=S\Sigma S^t$ for some $S\in SO(2m+2)$ and with
 \begin{equation}\label{eq:sigma}
 \Sigma
 =
 \left(
    \begin{array}{ccccc}
      0 & -x_1 & 0 & 0 &\cdots \\
      x_1 & 0 &0 & 0 &\cdots \\
     0 & 0 & 0 & -x_2 &\cdots \\
     0 & 0 & x_2 & 0 &\cdots \\
            && \vdots &&\ddots\\
    \end{array}
    \right)
\end{equation}
Set $c=S^ta$, $d=S^tb$ and note that $a^ta=c^tc$; $b^tb=d^td$,
$d^tc=b^ta=-1$.
A straightforward computation gives that
$$
f(a,b,B_0')=
\sum^{m+1}_{k=1}
\frac{(x_k(d_{2k-1}^2+d_{2k-1}^2)+(c_{2k-1}d_{2k }-c_{2k }d_{2k-1}))^2
+(c_{2k-1}d_{2k-1}+c_{2k }d_{2k})^2
}
{d_{2k-1}^2+d_{2k }^2} -(d_{2k-1}^2+d_{2k }^2).
$$
Clearly,
\begin{eqnarray*}
f(a,b,B_0')&\geq&
\sum^{m+1}_{k=1}\frac{(c_{2k-1}d_{2k-1}+c_{2k }d_{2k})^2}{\sum^{m+1}_{j=1} (d_{2j-1}^2+d_{2j }^2)} -\sum^{m+1}_{k=1}(d_{2k-1}^2+d_{2k }^2)
=-b^tb+\frac{1}{b^tb} \sum^{m+1}_{k=1} (c_{2k-1}d_{2k-1}+c_{2k }d_{2k})^2  \\
&\geq &
-b^tb+\frac{1}{(m+1)b^tb}  (\sum^{m+1}_{k=1} c_{2k-1}d_{2k-1}+c_{2k }d_{2k})^2
=-b^tb+\frac{1}{(m+1)b^tb}
\end{eqnarray*}
Let   $z\in\mathbb R^{2m+2}\setminus\{0\}$ with $R=|z|^2<\frac{1}{\sqrt{m+1}}$. Then $f(-\frac{z}{R},z,B_0')>0$ for any skew-symmetric matrix $B_0'$.
\end{proof}

\begin{lemma}
If $G_0$ is a Lie group from the list (b-f) above, then $G_1\rtimes G_0$ does not act transitively on $H^{2r+1}_{r}$.
\end{lemma}
\begin{proof}
Note that
$$SL(m+1,\mathbb C),\ SL(m+1,\mathbb C) U(1),\ SL(m+1,\mathbb C) S_a,  SU(m+1)S_a, SU(m+1)\mathbb C^* \subset GL(m+1,\mathbb C);$$
$$SL(m+1,\mathbb H), \ SL(m+1,\mathbb H) U(1), \ SL(m+1,\mathbb H) S_a,\ SL(m+1,\mathbb H) \mathbb C^*\subset GL(2(m+1),\mathbb C);$$
$$Sp(m+1,\mathbb C), \ Sp(m+1,\mathbb C) \mathbb C^*\ Sp(m+1,\mathbb C) U(1),\ Sp(m+1,\mathbb C) S_a\subset GL(2(m+1),\mathbb C);$$
Since $Spin(9,1)=SL(2,\mathbb O)$ (see Baez \cite{baez}), we have  $Spin(9,1), Spin(9,1)\mathbb R^*_+\subset GL(8,\mathbb C) $.
By lemma \ref{GL}, it follows that all these Lie groups do not act transitively on $H^{2r+1}_{r}$.
\end{proof}
\begin{lemma}
If $G_0\in\{SL(m+1,\mathbb H) Sp(1), SL(m+1,\mathbb H) \mathbb H^*, Sp(m+1)\mathbb H^*\}$, then $G_1\rtimes G_0$ does not act transitively on $H^{4m+3}_{2m+1}$.
\end{lemma}
\begin{proof}
Analogously to the proof of Lemma \ref{GL}, one can show that $G_1\rtimes SL(m+1,\mathbb H) \mathbb H^*$ does not act transitively on $H^{4m+3}_{2m+1}$. For the other two cases, the conclusion clearly follows from the inclusions $SL(m+1,\mathbb H) Sp(1),  Sp(m+1)\mathbb H^*\subset SL(m+1,\mathbb H) \mathbb H^*$.
\end{proof}
\begin{lemma}
(i) The semidirect product group $G_1\rtimes SO(m+1)\mathbb R^*_+$  acts transitively on $H^{2m+1}_{m}$ if and only if $G_1=U_{m+1}$.

(ii) the semidirect product group $G_1\rtimes Spin(7)\mathbb R^*_+$  acts transitively on $H^{15}_{7}$ if and only if $G_1=U_8$.

(iii) the semidirect product group $G_1\rtimes G_2\mathbb R^*_+$  acts transitively on $H^{13}_{6}$ if and only if $G_1\in\{U_7,\mathfrak g_2,ad_{Im(\mathbb O)}\}$.

\end{lemma}
\begin{proof} It is easy to see that $SO(m+1)\mathbb R_+^*$ does not act transitively on $H^{2m+1}_{m}$ and therefore $G_1$ is not trivial. Since $Spin(7)\subset SO(8)$ and $G_2\subset SO(7)$, we see that $Spin(7)\mathbb R^*_+$ and $ G_2\mathbb R^*_+$ do not act transitively.

(i) Since $SO(m+1)$ acts irreducibly on $\Lambda^{2}\mathbb R^{m+1}$, it follows that $G_1=U_{m+1}$. Now we show that $U_{m+1}\rtimes SO(m+1)\mathbb R^*_+$ acts transitively on $H^{2m+1}_{m}$.

Let $(p_1,p_2) \in H^{2m+1}_{m}$, with $p_1=(-1,0,\cdots,0)$ and $p_2=(1,0,\cdots,0)$. Let
$A\in SO(m+1)$, $\lambda\in R^*_+$ and $B_0'$ skew-symmetric and set
\begin{eqnarray}\label{eq:ab}
\left(
    \begin{array}{c}a \\ b\\
    \end{array}
    \right)
=
\left(
    \begin{array}{cc}
      I_{m+1} & B_0' \\
      0 & I_{m+1} \\
    \end{array}
  \right)
  \left(
    \begin{array}{cc}
      \lambda A & 0 \\
      0 &  \lambda^{-1} A \\
    \end{array}
  \right)
    \left(
    \begin{array}{c}p_1 \\ p_2\\
    \end{array}
    \right)
    .
\end{eqnarray}
We have  $a-B_0'b=\lambda Ap_1$ and $b=\lambda^{-1} Ap_2$. It follows that
\begin{eqnarray}\label{eq:hab}
h(a,b,B_0')=(a-B_0'b)^t(a-B_0'b)-\frac{(p_1^tp_1)(p_2^tp_2)}{b^tb}=0.
\end{eqnarray}
 Any skew-symmetric matrix $B'_0$ can be written as $B_0'=S\Sigma S^t$ for some $S\in SO(m+1)$ and with $\Sigma$ given by \eqref{eq:sigma}. Let  $N=(m+1)/2$ if $m$ is odd, or   $N=m/2$ if $m$ is even.

We shall need the following inequality
\begin{eqnarray}\label{ineq2}
\sum_{k=1}^N\frac{x_k^2}{y_k^2}\geq \frac{(\sum_{k=1}^Nx_k)^2}{\sum_{k=1}^Ny_k^2},
\end{eqnarray}
and the equality holds if and only if $\frac{x_1}{y_1^2}=\cdots =\frac{x_N}{y_N^2}.$

To simplify the computation, we introduce $c=S^ta$, $d=S^tb$. When $m$ is even, after modifying the matrix $S$, we may assume $d_{m+1}=0$.
By a straightforward computation, similar to the one in the proof of Lemma \ref{GL}, we get
\begin{eqnarray*}
h(a,b,B_0')&=&-\frac{(p_1^tp_1)(p_2^tp_2)}{\sum^{m+1}_{k=1}d_k^2}\\
&&+
\sum^{N}_{k=1}
\frac{(x_k(d_{2k-1}^2+d_{2k-1}^2)+(c_{2k-1}d_{2k }-c_{2k }d_{2k-1}))^2
+(c_{2k-1}d_{2k-1}+c_{2k }d_{2k})^2
}
{d_{2k-1}^2+d_{2k }^2}\\
&\geq& - \frac{1}{\sum^{m+1}_{k=1} d_k^2}
+  \frac{(\sum^{N}_{k=1}(c_{2k-1}d_{2k-1}+c_{2k}d_{2k}))^2}{\sum^{N}_{k=1}(d_{2k-1}^2+d_{2k}^2)}
= 0.
\end{eqnarray*}
By \eqref{eq:hab}, the last inequality must be an equality, thus
\begin{eqnarray}\label{eqc1d2}
\frac{ c_{ 1}d_{ 1}+c_{2  }d_{2 } }
{d_{ 1}^2+d_{2  }^2}=\cdots=\frac{ c_{2N-1}d_{2N-1}+c_{2N }d_{2N} }
{d_{2N-1}^2+d_{2N }^2}=-\frac{1}{b^tb}, \text{and} \
x_k=-\frac{ c_{2k-1}d_{2k }-c_{2k }d_{2k-1} }{ d_{2k-1}^2+d_{2k-1}^2 }.
\end{eqnarray}
By \eqref{eq:ab}, for our choice of $p_1$ and $p_2$, we get $c+\lambda^2d=\Sigma d$, which by \eqref{eqc1d2}, gives $\lambda^2=1/(b^tb)$.

To show the transitivity, we solve equation \eqref{eq:ab} for any $(a,b)\in H^{2r+1}_r$. Because of our particular choice of $p_1$ and $p_2$, the equation \eqref{eq:ab} is equivalent to $a-B_0'b=\lambda^2 b$ and $b=\lambda^{-1} Ap_2$. There exists $A\in SO(m+1)$ such that $b/\sqrt{b^tb}=Ap_2$. The equation $a-B_0'b=\lambda^2 b$ is equivalent to \eqref{eqc1d2} with $c=S^ta$, $d=S^tb$ and $S\in SO(m+1)$.
 Let $q_i$ be the row $i$ of the matrix $S^t=\{q_{ij}\}$. The system \eqref{eqc1d2} is equivalent to
\begin{eqnarray}\label{s:btq}
b^t(q_{2k-1}^tq_{2k-1})a+b^t(q_{2k }^tq_{2k})a
+\frac{1}{b^tb}(b^t(q_{2k-1}^tq_{2k-1})b+b^t(q_{2k }^tq_{2k})b)=0,
\end{eqnarray}
 for any $k\in\{1,\cdots,N\}$.
Since $S\in SO(m+1)$, we have $Trace(q_{i}^tq_{i})=1$.
 Replacing the entries $q_{(2k-1)1}^2=1- \sum_{j=2}^{m+1} q_{(2k-1)j}^2$ and $q_{(2k) 1}^2=1- \sum_{j=2}^{m+1} q_{(2k)j}^2$ of the matrices $q_{2k-1}^tq_{2k-1}$ and $q_{2k }^tq_{2k}$, the equation \eqref{s:btq} becomes linear in the variables $q_{(2k-1)1}$ and $q_{(2k)1}$. Clearly, there exists $S\in SO(m+1)$ such that \eqref{s:btq} holds for any $k\in\{1,\cdots,N\}$.
 It follows $U_{m+1}\rtimes SO(m+1)\mathbb R_+^*$ acts transitively on $H^{2m+1}_m$.

(ii) Since $Spin(7)\subset SO(8)$ acts irreducibly on $\Lambda^{2}\mathbb R^{8}$, it follows that $G_1=U_{8}$. Since $Spin(7)$ acts transitively on $S^7$, for any $b\in\mathbb R^8\setminus\{0\}$, there exists $A\in Spin(7)$ such that $b/\sqrt{b^tb}=Ap_2$. By (i), for any $(a,b)\in H^{15}_7$, there exists a skew-symmetric matrix $B'_0$ such that equation \eqref{eq:ab} holds.

(iii) The action of $G_2$ on $\Lambda^2\mathbb R^7=Lie(U_7)$ gives a decomposition of $\Lambda^2\mathbb R^7$ into two $G_2$-modules: $\Lambda^2\mathbb R^7
=so(Im(\mathbb O))
=\mathfrak g_2\oplus ad_{Im(\mathbb O)}
=Der(\mathbb O)\oplus ad_{Im(\mathbb O)}$ (see Baez \cite{baez}), with
$\dim\mathfrak g_2=14$ and
$\dim ad_{Im(\mathbb O)}=7$.
It follows $G_1\in\{U_7,\mathfrak g_2, ad_{Im(\mathbb O)}\}$.
Since $G_2$ acts transitively on $S^6$, for any $b\in\mathbb R^7\setminus\{0\}$, there exists $A\in G_2$ such that $b/\sqrt{b^tb}=Ap_2$.
By (i), for any $(a,b)\in H^{13}_6$, there exists $B'_0\in skew_7$ such that equation \eqref{eq:ab} holds.  It follows $U_{7}\rtimes G_2\mathbb R_+^*$ acts transitively on $H^{13}_6$. By a straightforward computation, we get that a such $B_0'$ can also be chosen in  $\mathfrak g_2$ or in $ad_{Im(\mathbb O)}$. Thus $\mathfrak g_2\rtimes G_2\mathbb R_+^*$ and $ad_{Im(\mathbb O)}\rtimes G_2\mathbb R_+^*$ act transitively on $H^{13}_6$.
\end{proof}
This completes the proof of Theorem \ref{t:3.1}.
\end{proof}

Unlike in the real pseudo-hyperbolic case, the groups
   $SU(m-s,s+1)$, $Sp(m-s,s+1)$, $Sp^\pi(m+1)$, $SU^\pi(m+1)$, $Sp^\pi(m+1)$, $Sp(m-s,s+1)$ and  $Sp^\pi(m+1)$,
  act only almost effectively on
  $\mathbb CH^{m}_{s}$, $\mathbb  CH^{2m+1}_{2s+1}$, $\mathbb  CH^{2m+1}_{m}$, $\mathbb AP^m$,
  $\mathbb AP^{2m+1}$, $\mathbb HH^{m}_{s}$ and $\mathbb BP^m$, respectively. In order to make these actions effective, one has to consider the action of the quotient of each  group by its center (see \cite[\S 7.12 Note on effectivity]{bes}).
Let $Z_{m+1}=\{\exp(2\pi ik/(m+1)\ |\ k=0,\cdots,m\}$. Note that
 $$Z(SU(m-s,s+1))=Z_{m+1}, \ \ Z(Sp(m-s,s+1))=Z_2,$$ $$ Z(Sp^\pi(m+1))=Z(Sp(m,\mathbb R))=Z_{2},$$
 $$Z(SU^\pi(m+1))=Z(SL(m+1,\mathbb R))=\{x\in\mathbb R\ | \ x^{m+1}=1 \}.$$

\begin{theorem}
   Let $G$ be a connected Lie group acting completely reducible on $\mathbb R^{2n+2}_{2r+2}$. One of the following holds:
   \begin{enumerate}
     \item $G$ is a closed subgroup of $SO_0(2n-2r,2r+2)$ acting on $\mathbb CH^{n}_{r}$,
     \item $G$ is a closed subgroup of $SO_0(n+1,n+1)$ acting on $\mathbb AP^{n}$,
     \item $G$ is a closed subgroup of $SO_0(4m-4s,4s+4)$ acting on $\mathbb HH^{m}_{s}$,
     \item $G$ is a closed subgroup of $SO_0(2m+2,2m+2)$ acting on $\mathbb BP^{m}$,
   \end{enumerate}
   and the action is   effective and transitive  if and only if $G$ is contained in   Table 2.
\begin{center}
 Table 2
\end{center}
\begin{center}
  \begin{tabular}{| c | l | l | c |}
  \hline
   No.  & G             & H           & $G/H$
   \\  \hline
    (1) & $SU(m-s,s+1)/Z_{m+1}$    & $S(U(1)U(m-s,s))/Z_{m+1}$ & $\mathbb CH^{m}_{s}$\\\hline
    (2) & $Sp(m-s,s+1)/Z_{2}$    & $U(1)Sp(m-s,s)/Z_{2}$ & $\mathbb  CH^{2m+1}_{2s+1}$\\\hline
    (3)& $Sp^\pi(m+1)/Z_{2}$    & $Sp^\pi(m)  U(1)/Z_{2}$  & $\mathbb  CH^{2m+1}_{m}$ \\\hline
    (4) &
     \begin{tabular}{lr}
         $SU^\pi(m+1)/Z_2$,  & if $m$ is odd \\
         $SU^\pi(m+1)$,   & if $m$ is even
      \end{tabular}
       &
    \begin{tabular}{lr}
      $S(U^\pi(m)U^\pi (1))/Z_2$,  & if $m$ is odd\\
      $S(U^\pi(m)U^\pi (1))$,& if $m$ is even
    \end{tabular}
       & $\mathbb AP^m$\\\hline
    (5) & $Sp^\pi(m+1)/Z_2$ & $Sp^\pi(m)  U^\pi_0(1)/Z_2$  & $\mathbb AP^{2m+1}$ \\\hline
    (6) & $Sp(m-s,s+1) /Z_{2} $ & $Sp(m-s,s)  Sp(1)/Z_{2}$ & $\mathbb HH^{m}_{s}$\\\hline
    (7) & $ Sp^\pi(m+1)/Z_2$  & $ Sp^\pi(m) Sp^\pi(1)/Z_2$  & $\mathbb BP^m$\\ \hline
  \end{tabular}
\end{center}
\end{theorem}

\begin{proof} Clearly, if $G$ is one of the groups in Table 2, then $G$ acts effectively and transitively on the corresponding space.
  We shall prove that if $G$ is a subgroup satisfying (1-4) and acting transitively and effectively, then it is contained in   Table 2.
  From the assumption that $G$ acts completely reducible, we get $G$ is reductive.

  The transitivity of $G$ on $\mathbb CH^{n}_{r}=H^{2n+1}_{2r+1}/U(1)$ implies that   $G\,U(1)$ acts  transitively on $H^{2n+1}_{2r+1}$. There exists a complex structure $I$ on $\mathbb{R}^{2n+2}$ such that $I\in Z(G\,U(1))$ (e.g. take $I=i\mathrm{Id}_{2n+2}\in  G\,U(1)\subset SO_0(2n-2r,2r+2)$), and therefore,
  $G\,U(1)\subset U(m-s,s+1)$. By the transitivity of $G\,U(1)$ on $H^{2n+1}_{2r+1}=SU(n-r,r+1)U(1)/SU(n-r,r)U(1)$, we get
  $$su(n-r,r+1)+u(1)=su(n-r,r)+u(1)+\mathfrak{g}+u(1)=su(n-r,r)+u(1)+\mathfrak{g}.$$ It follows that $G$ acts transitively on
  $H^{2n+1}_{2r+1}=SU(n-r,r+1)U(1)/SU(n-r,r)U(1)$.  On the other hand, the effectivity of $G$ on  $\mathbb CH^{m}_{s}$ clearly implies that  $G$ acts also effectively on  $H^{2n+1}_{2r+1}$ and $G\cap U(1)=\{e\}$.
  Hence, by Table 1 of Theorem \ref{t:3.1},
  $G \in\{SU(m-s,s+1)/Z_{m+1} , Sp(m-s,s+1)/Z_{2} , Sp^\pi(m+1)/Z_{2} \}$.

  We now repeat the argument above for the other cases. The transitivity of $G$  on $\mathbb AP^n=H^{2n+1}_{n}/U^\pi_0(1)$, implies  the transitivity of $G U^\pi_0(1)$   on  $H^{2n+1}_{n}$. The existence of a para-complex structure $I$ on $\mathbb{R}^{2n+2}$ such that $I \in Z(G\,U^\pi_0(1))$, implies that $G U^\pi_0(1)$ is a subgroup of $ SU^\pi(n+1) U^\pi_0(1)$. It follows that
  $$sl(n+1)+\mathbb{R} =sl(n)+\mathbb{R} +\mathfrak{g}+\mathbb{R} =sl(n)+\mathbb{R}+\mathfrak{g}.$$
  Therefore, $G$ acts transitivity on $H^{2n+1}_n=GL_+(n+1,\mathbb{R})/GL_+(n,\mathbb{R})$.  Obviously, the  effectivity of $G$ on $\mathbb AP^n$ implies the  effectivity on $H^{2n+1}_{n}$ and $G\cap U^\pi_0(1)=\{e\}$.
    Hence, by Table 1,  $G$ falls in the cases (4-5) of  Table 2.

   Analogously, we get that if $G$ acts effectively and transitively on $\mathbb HH^{m}_{s}=H^{4m+3}_{4s+3}/Sp(1)$ or $\mathbb BP^m=H^{4m+3}_{2m+1}/Sp^\pi(1)$, then $G=Sp(m-s,s+1)/Z_2$ or $G=Sp^\pi(m+1)/Z_2$, respectively.
\end{proof}

\section{The proof of the main theorem}
\begin{proof}[Proof of Theorem \ref{t:main}]
     Let $(G,H)$  be a pair of Lie groups contained in  Tables 1 or 2. We denote by $\mathfrak{g},\mathfrak{h}$ their associated Lie algebras and by $\mathrm{ad}:\mathfrak{g}\to \mathrm{gl}(\mathfrak{g})$ the adjoint representation of $\mathfrak{g}$. When $\mathfrak{h}$ is not semisimple, then the isotropy representation  $\chi=\mathrm{ad}:\mathfrak h\to \mathrm{gl}(\mathfrak g)$  is completely reducible  simply because the center
$Z(\mathfrak h)\in\{U(1),U^\pi_0(1)\}$ acts by semisimple endomorphisms. When  $\mathfrak{h}$ is  semisimple, $ad:\mathfrak h\to \mathrm{gl}(\mathfrak g)$ is always completely reducible (see \cite[Theorem 6.3]{humphreys}).  It follows that there exits a subspace $\mathfrak m$ in $\mathfrak{g}$ such that $\mathfrak{g}=\mathfrak{h}\oplus\mathfrak{m}$ and $[\mathfrak{h},\mathfrak{m}]\subset\mathfrak{m}$. Such a homogeneous  space $G/H$ is called reductive.

Let $(\cdot,\cdot)$ be an $ad(\mathfrak{h})$-invariant symmetric non-degenerate bilinear form on $\mathfrak{m}$, associated to a $G$-invariant pseudo-Riemannian metric $g$ on $G/H$. Let $\mathfrak{m}=\mathfrak{m}_+\oplus\mathfrak{m}_-$ be an orthogonal decomposition of $\mathfrak{m}$ such that $(\cdot,\cdot)$ is positive definite on $\mathfrak{m}_+$ and negative definite on $\mathfrak{m}_-$. There exists a Cartan involution $T$ of $\mathfrak{g}$ such that $\mathfrak{m}_+\subset\mathfrak{g}_+$ and $\mathfrak{m}_-\subset\mathfrak{g}_-$, where
$$\mathfrak{g}_+=\{X\in\mathfrak{g}\,|\, T(X)=X\},\ \ \ \mathfrak{g}_-=\{X\in\mathfrak{g}\,|\, T(X)=-X\}.$$
As in \S \ref{s:irr}, changing the point where the isotropy is computed, we may assume that the isotropy Lie algebra $\mathfrak{h}$ is
$T$-invariant. We have $T(\mathfrak{h})=\mathfrak{h}$ and thus, $T(\mathfrak{m})=\mathfrak{m}$.
Letting  $\mathfrak{h}_\pm=\mathfrak{g}_\pm\cap\mathfrak{h}$, we note that $\mathfrak{h}=\mathfrak{h}_+\oplus\mathfrak{h}_-$.

Now, we define the compact forms $\mathfrak{g}^*=\mathfrak{g}_++i\mathfrak{g}_-$, $\mathfrak{h}^*=\mathfrak{h}_++i\mathfrak{h}_-$; let  $\mathfrak{m}^*=\mathfrak{m}_++i\mathfrak{m}_-$, and take $G^*$ and $H^*$ to be the connected analytic Lie groups, with $\mathrm{Lie}(G^*)=\mathfrak{g}^*$ and $\mathrm{Lie}(H^*)=\mathfrak{h}^*$. Clearly, $G^*/H^*$ is a compact homogeneous space, and the associated bilinear form $(\cdot,\cdot)^*$ on $\mathfrak{m}^*$ is positive definite and its associated $G^*$-invariant metric $g^*$ is Riemannian (see \cite{kath1} for the definition of $(\cdot,\cdot)^*$).  Moreover, $\mathfrak{m}_+^*$ and $\mathfrak{m}_-^*$ are orthogonal to each others with respect to $(\cdot,\cdot)^*$. It means that $(\mathfrak{g},\mathfrak{h},\mathfrak{m},(\cdot,\cdot) ) $
is a $T$-dual to $(\mathfrak{g}^*,\mathfrak{h}^*,\mathfrak{m}^*,(\cdot,\cdot)^* ) $ (see Kath \cite[Definition 3.1]{kath1}).

 The compact dual triples $(G^*,H^*,G^*/H^*)$ of  all triples $(G,H,G/H)$ of Tables 1 and 2, with a non-compact  $G$, are listed in the next table.
{\small
\begin{center}
 Table 3
  \begin{tabular}{| c | l | l | c | l | l | c |}
  \hline
   No.  & $G$             & $H$           & $G/H$ & $G^*$             & $H^*$           & $G^*/H^*$
   \\    \hline
    (1) & $SO_0(n-r,r+1)$ & $SO_0(n-r,r)$ & $H^n_{r}$& $SO (n+1)$ & $SO (n )$ & $S^n $ \\\hline
    (2) & $G_2^*$  & $SU(2,1)$  & $H^6_2$& $G_2$  & $SU(3)$  & $ S^6$\\\hline
    (3) & $G_2^*$  & $ SL(3,\mathbb R)$  & $H^6_3$& $G_2$  & $SU(3)$  & $ S^6$\\\hline
    (4) & $Spin_0(4,3)$  & $G_2^*$  & $H^7_3$& $Spin(7)$  & $G_2$  & $ S^7$\\\hline
    (5) & $SU(m-s,s+1)$ & $SU(m-s,s)$ & $H^{2m+1}_{2s+1}$& $SU(m +1)$     & $SU(m )$  & $S^{2m+1}$\\\hline
    (6) & $ U(m-s,s+1)$     & $ U(m-s,s)$  & $H^{2m+1}_{2s+1}$ & $ U(m +1)$     & $ U(m )$  & $S^{2m+1} $\\\hline
    (7) & $ SU^\pi(m+1)$  & $ SU^\pi(m)$  & $H^{2m+1}_m$ & $ SU (m+1)$  & $ SU (m)$  & $S^{2m+1} $\\\hline
    (8) & $ U^\pi_0(m+1)$  & $ U^\pi_0(m)$  & $H^{2m+1}_m$& $ U(m +1)$     & $ U(m )$  & $S^{2m+1} $\\\hline
    (9) & $Spin_0(8,1)$  & $Spin(7)$  & $H^{15}_{7}$& $Spin(9)$  & $Spin(7)$  & $ S^{15}$\\\hline
    (10) & $Spin_0(5,4)$  & $Spin_0(4,3)$  & $H^{15}_{7}$& $Spin(9)$  & $Spin(7)$  & $ S^{15}$\\\hline
    (11) & $Sp(m-s,s+1)$ & $Sp(m-s,s)$ & $H^{4m+3}_{4s+3}$& $Sp(m +1)$ & $Sp(m)$ & $S^{4m+3} $\\\hline
    (12) & $Sp(m-s,s+1)U(1)$  & $Sp(m-s,s)  U(1)$  & $H^{4m+3}_{4s+3}$&$Sp(m +1)U(1)$  & $Sp(m )  U(1)$  & $S^{4m+3} $\\\hline
    (13) & $Sp(m-s,s+1) Sp(1)$ & $Sp(m-s,s)  Sp(1)$ & $H^{4m+3}_{4s+3}$& $Sp(m +1) Sp(1)$ & $Sp(m )  Sp(1)$ & $S^{4m+3} $\\\hline
    (14) & $ Sp^\pi(m+1)$  & $ Sp^\pi(m)$  & $H^{4m+3}_{2m+1}$& $Sp(m +1)$ & $Sp(m)$ & $S^{4m+3} $ \\\hline
    (15) & $ Sp^\pi(m+1) U^\pi_0(1)$  & $ Sp^\pi(m)  U^\pi_0(1)$  & $H^{4m+3}_{2m+1}$ & $ Sp (m+1)U (1) $  & $ Sp (m)  U (1)$  & $S^{4m+3} $ \\\hline
    (16) & $ Sp^\pi(m+1)  Sp^\pi(1)$ & $ Sp^\pi(m)  Sp^\pi(1)$  & $H^{4m+3}_{2m+1}$& $Sp(m +1) Sp(1)$ & $Sp(m )  Sp(1)$ & $S^{4m+3} $\\\hline
    (17) & $SU(m-s,s+1)$    & $S(U(m-s,s)U(1))$ & $\mathbb CH^{m}_{s}$ & $SU(m+1) $    & $S(U(m)U(1)) $ &
    $\mathbb CP^{m} $\\\hline
    (18) & $Sp(m-s,s+1)$    & $Sp(m-s,s)U(1)$ & $\mathbb  CH^{2m+1}_{2s+1}$& $Sp(m +1)$    & $Sp(m )U(1)$ & $\mathbb  CP^{2m+1} $\\\hline
    (19) & $SU^\pi(m+1) $ & $S(U^\pi(m)U^\pi (1))$ & $\mathbb AP^m$ & $SU (m+1) $ & $S(U (m)U (1))$ & $\mathbb CP^m$\\ \hline
    (20) & $Sp^\pi(m+1) $ & $Sp^\pi(m)  U^\pi_0(1) $  & $\mathbb AP^{2m+1}$ & $Sp (m+1)$  & $Sp (m) U(1) $  & $\mathbb  CP^{2m+1} $\\\hline
    (21) & $Sp(m-s,s+1) $ & $Sp(m-s,s)  Sp(1) $ & $\mathbb HH^{m}_{s}$& $Sp(m +1) $ & $Sp(m)  Sp(1) $ & $\mathbb HP^{m} $\\\hline
    (22) & $ Sp^\pi(m+1) $  & $ Sp^\pi(m) Sp^\pi(1) $  & $\mathbb BP^m$& $Sp(m +1) $ & $Sp(m)  Sp(1) $ & $\mathbb HP^{m} $\\ \hline
  \end{tabular}
\end{center}
}

By Kath \cite[Corollary 4.1]{kath1}, the $G$-homogeneous Einstein pseudo-Riemannian metrics on $G/H$ are in one-to-one correspondence to
the  $G^*$-homogeneous Einstein Riemannian metrics on $G^*/H^*$. Thus, by Ziller's classification of homogeneous Einstein Riemannian metrics on sphere and projective spaces (see Ziller \cite{ziller}), we get the following:
 \begin{enumerate}
   \item[(i)]   for the cases (1-8) of Table 3,  the only $G$-homogeneous Einstein pseudo-Riemannian is the constant curvature metric,
   \item[(ii)] for each of cases (9-11, 14) of Table 3,  we have only two $G$-homogeneous Einstein pseudo-Riemannian metrics: the constant curvature one and the Einstein metric of the canonical variation,
   \item[(iii)] for each of (18,   20),  we have only two $G$-homogeneous Einstein pseudo-Riemannian on $G/H$,
   \item[(iv)] the cases (12-13) are special cases of (11), and the cases (15-16) are special cases of (14),
   \item[(v)]  for the cases (17, 19, 21, 22), we have only one $G$-homogeneous Einstein pseudo-Riemannian on $G/H$.
 \end{enumerate}
\end{proof}
\begin{remark}
  We recall from Ziller \cite{ziller} that the homogeneous Einstein Riemannian metrics on $S^{4n+3}$ (associated to the canonical variation of the Hopf fibration $S^{4n+3}\to\mathbb{C}P^n$) are normal homogeneous, but the homogeneous Einstein Riemannian metrics on $S^{15}$ (associated to  the Hopf fibration $S^{15}\to S^8$) and on $\mathbb{C}P^{2n+1}$ are not even naturally reductive. Since the notions of normal homogeneity and natural reductivity are preserved under duality, it follows that 2  homogeneous Einstein metrics on $H^{15}_7$, namely  the $Sp(2,2)$ and $Sp^{\pi}(4)$-invariant metrics, and the Einstein metrics on $H^{4m+3}_{4s+3}$ are normal homogeneous, but the non-canonical homogeneous Einstein metrics on $\mathbb{C}H^{2m+1}_{2s+1}$, $\mathbb{A}P^{2m+1}$
   and the other 2 non-canonical Einstein metrics on $H^{15}_7$ (the $Spin_0(5,4)$ and $Spin_0(8,1)$-invariant metrics)
   are not naturally reductive.
\end{remark}

\begin{remark}
Our classification of Einstein homogeneous pseudo-Riemannian on $H^n_r$ is obtained in the case when $G$ acts completely reducible on $\mathbb R^{n+1}$, which is essentially the case of a reductive Lie group $G$. We plan to compute somewhere else the Einstein homogeneous pseudo-Riemannian on $H^n_r$  for the remaining cases (17-28) of Table 1, which are the solvmanifold cases.
\end{remark}
\begin{acknowledgments}
This work was partially supported by a grant of the Romanian National
Authority for Scientific Research, CNCS - UEFISCDI, project number
PN-II-ID-PCE-2011-3-0362.
\end{acknowledgments}

\bibliographystyle{amsplain}
\providecommand{\bysame}{\leavevmode\hbox to3em{\hrulefill}\thinspace}
\providecommand{\MR}{\relax\ifhmode\unskip\space\fi MR }
\providecommand{\MRhref}[2]{%
  \href{http://www.ams.org/mathscinet-getitem?mr=#1}{#2}
}
\providecommand{\href}[2]{#2}

\end{document}